\documentclass[12pt,reqno]{amsart}

\usepackage{graphicx}
\usepackage{fullpage}
\usepackage{etoolbox}
\usepackage[inline]{enumitem}
\usepackage[utf8]{inputenc}
\usepackage{amsmath}  
\usepackage{amssymb}     
\usepackage{amsthm}  
\usepackage{bbm}    
\usepackage{bm}
\usepackage{accents}
\usepackage{mathrsfs}
\usepackage{mathtools}
\usepackage{fixmath}
\usepackage{fullpage}
\usepackage{tikz}
\usepackage{microtype}
\usepackage{standalone}
\usepackage{cite}
\usepackage{float}
\usepackage{cases}
\usepackage[colorlinks,linkcolor=blue,citecolor=magenta]{hyperref}
\usetikzlibrary{patterns.meta}

\newtheorem{theorem}{Theorem}[section]
\newtheorem{proposition}[theorem]{Proposition}

\newtheorem{lemma}[theorem]{Lemma}

\theoremstyle{definition}

\renewcommand{\geq}{\geqslant}
\renewcommand{\leq}{\leqslant}
\newcommand{\conv}{\operatorname{conv}}

\def\ds{\displaystyle}
\def\R{\mathbb{R}}
\def\x{\boldsymbol{x}}

\title{Bucket Brigades: Uniqueness of the Fixed Point and Three-Worker Asymptotics}

\author{Yasser Alghouass}
\address{Y. Alghouass, École polytechnique, Institut Polytechnique de Paris, Palaiseau, France }
\email{yasser.alghouass@polytechnique.edu}

\author{Abderhamane Driouch}
\address{A. Driouch, École polytechnique, Institut Polytechnique de Paris, Palaiseau, France }
\email{abderrahmane.driouch@polytechnique.edu}

\author{Mohammed Lagmah}
\address{M. Lagmah, École polytechnique, Institut Polytechnique de Paris, Palaiseau, France }
\email{mohammed.lagmah@polytechnique.edu}

\author{Frédéric Meunier}
\address{F. Meunier, CERMICS, ENPC, Institut Polytechnique de Paris, Marne-la-Vallée, France}
\email{frederic.meunier@enpc.fr}

\subjclass[2020]{37N40, 90B30}

\keywords{bucket brigades; asymptotic behavior; production line}

\begin{document}

\begin{abstract}
A standard organization of production lines exhibiting self-balancing behavior is given by bucket brigades. Their study in operations research was initiated by the foundational work of Bartholdi and Eisenstein ({\em Operations Research}, 1996), where a simplified version of the model is considered. Their main result shows that when workers are ordered from the slowest to the fastest, the system is stable and converges to a ``fixed point,'' where each worker oscillates between two limiting positions. They also observe that the dynamics can become highly complex when this ordering condition is not satisfied. The {\em no-station} setting, in which work is distributed continuously and uniformly along the production line, is given special attention in their work. In a subsequent paper with Bunimovich ({\em Operations Research}, 1999), they characterize all stable behaviors of this setting for up to three workers.

In this work, we extend their analysis for three workers beyond the stable regime, providing a complete description when workers are ordered from the fastest to the slowest. We also show that, due to their restrictive notion of stability, some of their conclusions must be revisited. Finally, for an arbitrary number of workers, we prove that the fixed point is always unique in the no-station setting.
\end{abstract}

\maketitle

\section{Introduction}

\subsection{Context}
``Bucket brigades'' form a way to organize workers on a production line. In its simplest form, such an organization can be described as follows. Assume that the workers are numbered from $1$ to $n$ according to their sequence on the line. In the ideal version of the organization, the workers follow two rules:

\smallskip

    {\em Forward rule.} Worker $i$ remains devoted to a single item, processes it along the production line until the item is taken over by worker $i+1$ (or until worker $i$ reaches the end of the production line in case $i=n$), and then applies the backward rule. If worker $i$ catches up with worker $i+1$, then worker $i$ adjusts his velocity to that of worker $i+1$.

\smallskip
    
    {\em Backward rule.} Worker $i$ walks back to workers $i-1$ and takes over its item (or walks back to the start of the line and begins a new item in case $i=1$), and then applies the forward rule.

\smallskip

Moreover, no overtaking is possible.

Bucket brigades are used in a variety of industries~\cite{armbruster2006bucket,bratcu2005survey} (often with some flexibility in the way the idealized organization described above is applied). They have been studied from various points of view; see, e.g., \cite{armbruster2007bucket,bartholdi2001performance,bukchin2025sequencing,koo2009use,Wang2022Stochastic}. Bartholdi and Eisenstein~\cite{Bartholdi1996APL} have provided the first theoretical study, which remains the reference work on this topic. They show that even though their model is deterministic, it already captures a substantial part of the essence of this organization.  Under mild assumptions, they establish a variety of results, such as the existence of a ``fixed point,'' i.e., an initial position such that the workers are sent to the same position each time worker $n$ reaches the end of the production line. They also prove uniqueness of this fixed point when the workers are ordered from the slowest to the fastest.
Bartholdi, Bunimovich, and Eisenstein~\cite{Bartholdi1999} add to this foundational work a paper focusing on the two- and three-worker cases with fixed velocities, in which they aim at describing all possible asymptotic behaviors. Actually, as it is explicitly stated in their paper, they only care about ``stable'' behaviors, since only such behaviors can be observed in practice, and most of the arguments are informal or based on extensive experiments. The two-worker case has been completely solved by Gurumoorthy, Banerjee, and Paul~\cite{gurumoorthy2009dynamics} (especially thank to their introduction of the ``critical point,'' whose existence for more than two workers remain elusive).

In the original model of Bartholdi and Eisenstein~\cite{Bartholdi1996APL}, the production line is moreover subdivided into ``stations,'' whose presence further complicates the system's dynamics. As noted in their paper, the special setting where the stations are of negligible length, which we call {\em no-station}, is already of great interest. This is actually the situation considered in the paper by Bartholdi, Bunimovich, and Eisenstein~\cite{Bartholdi1999}.

Our contributions focus on the no-station setting and are twofold. First, we show that in this setting the fixed point is always unique, even when the workers are not ordered from the slowest to the fastest; the proof is remarkably short. Second, we complete the analysis of  Bartholdi, Bunimovich, and Eisenstein in the three-worker case. We provide proofs for several behaviors they identified, and we also investigate non-stable behaviors. This broader perspective shows that their neat conclusions do not necessarily extend beyond the stability assumption.
Moreover, we point out a limitation in their treatment of stability. They write at the very beginning of their paper:
\begin{quote}
By ``stable behavior'' we mean qualitative structure that persists, even in the presence of perturbations. This is the behavior that will assert itself in practice.
\end{quote}
However, their analysis is restricted to periodic behaviors, whereas we show that other qualitative structures may persist in a robust way without being periodic.

\subsection{Model} 

\subsubsection{No-station setting} We describe the model of Bartholdi and Eisenstein~\cite{Bartholdi1996APL} directly in the no-station setting. The state of the system can be represented as a vector $\x = (x_1,x_2,\ldots,x_n)$, with $0 \leq x_1 \leq x_2 \leq \cdots \leq x_n \leq 1$. The quantity $x_i$ represents the position of workers $i$ along the production line identified with the interval $[0,1]$. Each worker $i$ comes with a function $v_i \colon [0,1] \to \R_{>0}$: when worker $i$ has not caught up worker $i+1$, his velocity at position $x$ is $v_i(x)$. Otherwise, his velocity is equal to the velocity of worker $i+1$. The following assumptions are made for $v_i$: 
\begin{itemize}
    \item $v_i$ is continuous almost everywhere;
    \item there exist numbers $b$ and $B$ such that $0 < b < v_i(x) < B < +\infty$ for all $x \in [0,1]$.
\end{itemize}
Finally, the workers are assumed to move backward at an infinite velocity. This means that when worker $n$ reaches the end of the production line, then at the same instant, worker $n$ takes over from worker $n-1$, who takes over from worker $n-2$, and so on, and worker $1$ are sent back to the start of the production line to begin a new item. This corresponds to a {\em reset}.

\subsubsection{Equations}\label{subsubsec:equations}
The previous paragraph is unambiguous and provides an accurate description of the system's dynamics. We can however be formal and write explicit equations, as we show now. (The content of this subsection is for completeness only, and not used elsewhere in the paper.)

Denote by $\x^{(0)}$ the initial position of the workers (after the $0$th reset), and by $\x^{(k)}$ their position just after the $k$th reset. The dynamics of the system between the $k$th and $(k+1)$th resets can be written as follows. Let's see the position $\x=(x_1,x_2,\ldots,x_n)$ of the workers between the two resets as a function of the time $t \geq 0$. Then $\x(0)=\x^{(j)}$, and while $x_n(t) < 1$, the following relations hold: 
\[
\dot x_n = v_n(x_n) \quad \text{and} \quad
\dot x_i = \left \{ \begin{array}{ll} v_i(x_i) & \text{if $x_i < x_{i+1}$} \\[1ex]
\min(v_i(x_i),\dot x_{i+1}) & \text{if $x_i = x_{i+1}$}
\end{array}\right. 
\quad \text{for $i \in [n-1]$.}
\]
(Note that when $x_i = x_{i+1}$, if the minimum is attained only by the first term, this equality can hold only instantaneously, since it immediately ceases to be valid.) From this position of the workers seen as function of the time, we can recover the time $t^{(k)}$ elapsed between the $k$th and $(k+1)$th resets: it is the smallest $t$ such that $x_n(t) = 1$. From this, the reset can be expressed as $x_{i+1}^{(k+1)} = x_i(t^{(k)})$ for $i \in [n-1]$ and $x_1^{(k+1)} = 0$.

\subsubsection{Reset function}

A key object to study the dynamics of the system is the \emph{reset function} that, given the initial positions of the workers and assuming the first worker is located at the origin, returns the positions of the workers after the next reset. Note that, after this reset, the first worker is again located at the origin and that the behavior of the system is completely described by the iterates of the reset function.

Let $\triangle \coloneqq \{(x_2,x_3,\ldots,x_n) \colon 0 \leq x_2 \leq x_3 \leq \cdots \leq x_n \leq 1\}$. The positions of the workers after any reset is an element of $\triangle$ (up to removing the first component $x_1=0$). The whole behavior of the system is determined by the initial state. In particular, with an initial state of the form $(0,x_2,x_3,\ldots,x_n)$, the positions occupied by the workers just after the first reset is completely determined and of the form $(0,x'_2,x'_3,\ldots,x'_n)$. The reset function is the self-map of $\triangle$ mapping $(x_2,x_3,\ldots,x_n)$ to $(x'_2,x'_3,\ldots,x'_n)$. 
It is continuous~\cite[Appendix~A.1]{Bartholdi1996APL}.

\subsection{Main contributions}
As proved by Bartholdi and Eisenstein~\cite[Theorem~1]{Bartholdi1996APL}, the reset function always admits a fixed point, i.e., there always exists an initial position such that the workers are sent to the same position after each reset. Lemma~1 of their paper ensures that the fixed point is unique, provided that the workers are comparable with respect to their velocities, and that they are sequenced from the slowest to the fastest. (Worker $j$ is {\em faster} than worker $i$ if $\sup_{x\in [0,1]} \frac {v_i(x)} { v_j(x)} < 1$.)

Our first contribution shows that this condition is unnecessary when there are no station.

\begin{theorem}\label{thm:uniq-fixed-point}
    In the no-station setting, the fixed point of the reset function is always unique.
\end{theorem}

We provide the proof of this theorem in Section~\ref{subsec:proof-unique}.



Our other contributions concern the three-worker case, with constant velocities. When the velocities are constant, i.e., for each $i$, there exists $v_i \in \R_{>0}$ such that $v_i(x) \coloneqq v_i$ for all $x \in [0,1]$, the unique fixed point of the reset function is
    \[
    \frac 1 {\sum_{i=1}^n v_i} \Bigl(v_1,v_1+v_2,\ldots, v_1 + v_2 + \cdots + v_{n-1} \Bigr) \, ,
    \]
as it can be readily checked.

Bartholdi, Bunimovich, and Eisenstein~\cite{Bartholdi1999} study thoroughly the three-worker case with constant velocities. They subdivide $\R_{>0}^2$ into four regions, and describe the asymptotic behavior according to the region in which $(v_1/v_3,v_2/v_3)$ lies. These four regions are:
\[
\begin{array}{@{}l@{\quad\,}l@{}}
\text{Region 1} \coloneqq \{(x,y)\colon 0 < x \leq 1, \, 0 < y \leq x + 1 \} &
\text{Region 2} \coloneqq \{(x,y)\colon 0 < x \leq 1, \, x + 1 < y \} \\
\text{Region 3} \coloneqq \{(x,y)\colon 1 < x , \, 1 < y \} &
\text{Region $k$} \coloneqq \{(x,y)\colon 1 < x , \, 0 < y \leq 1\}\, .
\end{array}
\]
(The placement of the large inequalities are from us since they are a bit imprecise about the boundary of the regions.) They mention in the introduction that they are only interested in stable behaviors. By stable behaviors, they mean periodic trajectories; see the discussion in Section~1 of their paper. 

For Region~1, they claim that there is convergence to the fixed point. It is not completely clear from the reading whether stability is required or not, and the way they got this conclusion is also unclear; see~Section~\ref{subsec:region1}. Concerning Region~2, they state that ``. . . the positions of the workers after walkbacks eventually alternates between $(0,v_1/(v_1+v_3),1)$ and $(0,0,v_1/(v_1+v_3))$.'' This is only correct when stability is assumed.

\begin{proposition}\label{prop:region2}
    For every choice of constant velocities satisfying simultaneously $v_1 < v_3$ and $v_1+v_3 \leq v_2$, there exist infinitely many initial positions for which there is convergence to the fixed point.
\end{proposition}

Experimentally, it seems that the only other possible behavior is the eventual alternation between $(0,v_1/(v_1+v_3),1)$ and $(0,0,v_1/(v_1+v_3))$ (the one asserted by Bartholdi, Bunimovich, and Eisenstein), but a proof remains elusive.

Concerning Region 3, they state that ``The positions of the workers after walkbacks eventually alternate between $(0,1,1)$, $(0,0,1)$, and $(0,0,0)$.'' This is not correct in full generality as shown by the following theorem, where we have set
\[
\theta \coloneqq \frac{v_1(v_2-v_3)}{v_1v_2-v_3^2} \qquad \text{and} \qquad \varphi \coloneqq \frac{v_2(v_1-v_3)}{v_1v_2-v_3^2} \, .
\]
(The quantity $\theta$ plays a special role in the study of the three-worker case.) Note that it considers the case when the workers are ordered from the fastest to the slowest.

\begin{theorem}\label{thm:region3}
  For every choice of constant velocities such that $v_3 \leq v_2 < v_1$, there are initial positions for each of the following asymptotic behaviors of the system:
    \begin{enumerate}[label=\textup{(\roman*)}]
    \item\label{fixed} it remains at the fixed point.
    \item\label{cycle} it reaches after finitely many resets the $3$-cycle $(0,1,1)$, $(0,0,1)$, $(0,0,0)$.
    \item\label{cycle-other} it reaches after finitely many resets the $3$-cycle $(0,\varphi,\varphi)$, $(0,\theta,1)$, $(0,0,\theta)$.
    \end{enumerate}
    There is no other asymptotic behavior.
\end{theorem}

Moreover, the asymptotic behavior in Region~3 can be significantly more complex when $v_1 < v_2$, as discussed in Section~\ref{sec:compl}. In particular, we identify a set of positions of positive measure that is stable in the following sense: whenever the initial positions lie in this set, the system remains in it for all time, and this property persists under small perturbations. Interestingly, within this set we exhibit very large periodic cycles, whereas Bartholdi, Bunimovich, and Eisenstein give the impression that such cycles can only occur in Region~$k$.


Concerning Region~$k$, the conclusions of the paper by Bartholdi, Bunimovich, and Eisenstein are mainly experimental and seem to remain correct even without assuming stability. However, in the section about this region, the authors claim ``. . . unlike the other regions, the asymptotic behavior can depend not only on the value of $(v_1/v_3,v_2/v_3)$ but also on the initial positions of the workers.'' As shown by Proposition~\ref{prop:region2} and Theorem~\ref{thm:region3}, this dependence to the initial positions also hold in Regions~2 and~3.

\section{Proofs}

\subsection{Proof of Theorem~\ref{thm:uniq-fixed-point}}\label{subsec:proof-unique}

\begin{proof}
    Consider two fixed points $(\bar x_2,\bar x_3, \ldots,\bar x_n)$ and $(x_2^\star,x_3^\star,\ldots,x_n^\star)$.

    We start by proving that $\bar x_n = x_n^\star$. Suppose thus for a contradiction that $\bar x_n \neq x_n^\star$. W.l.o.g., $\bar x_n < x_n^\star$. Denote by $\bar T$ and $T^\star$ the time between two resets for the first and second fixed points, respectively. We have $\bar T > T^\star$ (no infinite velocity). Since, at a fixed point, worker $2$ starts from the last position of worker $1$, the velocity of worker $1$ between two resets is always given by $v_1(x)$. The inequality $\bar T > T^\star$ implies thus that $\bar x_2 > x_2^\star$. Similarly, worker $3$ starts from the last position of worker $2$, which makes that the velocity of worker $2$ between two resets is given by $v_2(x)$. The inequality $\bar T > T^\star$ and $\bar x_2 > x_2^\star$ imply then that $\bar x_3 > x_3^\star$. And so on. Hence, $\bar x_n > x_n^\star$, a contradiction.

    Therefore $\bar x_n = x_n^\star$. The time between two resets is then the same for the two fixed points. As already used in the previous paragraph, at a fixed point, worker $i+1$ starts from the last position of worker $i$, and the velocity of worker $i$ between two resets is always given by $v_i(x)$. Applying this for worker $1$ gives $\bar x_2 = x_2^\star$, and repeating this argument successively for all workers shows that the two fixed points are actually identical.
\end{proof}

\subsection{Proofs of the three-worker results}\label{subsec:three-workers}

In the three-worker case, the reset function is a self-map of $\triangle = \{(x,y)\colon 0 \leq x \leq y \leq 1\}$, where $x$ has to be interpreted as the position of the second worker, and $y$ as the position of the third worker. In this section, we focus on the case where there are three workers and the velocities are constant.


Let $r_1 \coloneqq v_1/v_3$ and $r_2 \coloneqq v_2/v_3$, and consider the self-map of $\triangle$ defined by
\begin{numcases}{(x,y) \longmapsto}
\begin{pmatrix} 1 \\ 1 \end{pmatrix} & \text{if $\ds{1 \leq \min\bigl(x+(1-y)r_2,(1-y)r_1\bigr).}$ \label{first}}\\[1ex]
\begin{pmatrix}
    0 & -r_1 \\ 0 & 0
\end{pmatrix}
\begin{pmatrix}
    x \\ y
\end{pmatrix} + 
\begin{pmatrix}
    r_1 \\ 1
\end{pmatrix}
& \text{if $\ds{(1-y)r_1 < 1\leq x+ (1-y)r_2.}$}\label{second}
\\[1ex]
\begin{pmatrix}
    1 & -r_2 \\ 1 & -r_2
\end{pmatrix}
\begin{pmatrix}
    x \\ y
\end{pmatrix} + 
\begin{pmatrix}
    r_2 \\ r_2
\end{pmatrix}
& \text{if $\ds{x+(1-y)r_2 < \min\bigl(1,(1-y)r_1\bigr).}$ }\label{third} \\[1ex]
\begin{pmatrix}
    0 & -r_1 \\ 1 & -r_2
\end{pmatrix}
\begin{pmatrix}
    x \\ y
\end{pmatrix} + 
\begin{pmatrix}
    r_1 \\ r_2
\end{pmatrix}
& \text{if $\ds{(1-y)r_1 \leq x + (1-y)r_2  < 1.}$\label{fourth}}
\end{numcases}
The four cases~\eqref{first},~\eqref{second},~\eqref{third}, and~\eqref{fourth} partition $\triangle$ into four cells.

\begin{lemma}\label{lem:expr-reset}
   When there are three workers and the velocities are constant, the reset function is exactly the map given above.
\end{lemma}

\begin{proof}
We consider each case in turn.

    Case~\eqref{first}. Workers~$1$ and~$2$ catch up with worker $3$. Just before reset, the three workers are located at position~$1$, and thus, just after reset, worker~$1$ is sent back to position~$0$, while the two others stay at position~$1$.

   Case~\eqref{second}. Worker~$1$ does not catch up with worker~$3$ and that worker~$2$ catches up with worker~$3$. Worker~$1$ moves thus at constant velocity~$v_1$ and reaches position $(1-y)r_1$ just before reset. The expression of the positions after reset is then immediate.

    Case~\eqref{third}. Worker~$1$ catches up with worker~$2$ and that worker~$2$ does not catch up with worker~$3$. Worker~$2$ moves thus at constant velocity~$v_2$ and reaches with worker~$1$ the position $x+(1-y)r_2$ just before reset. The expression of the positions after reset is then immediate.

    Case~\eqref{fourth}. Worker~$1$ does not catches up with worker~$2$ (except maybe just before reset), and that worker~$2$ does not catch up with worker~$3$ either. Thus, worker~$1$ moves at constant velocity~$v_1$ and worker~$2$ moves at constant velocity~$v_2$. They reach respectively positions $(1-y)r_1$ and $x+(1-y)r_2$ just before reset. The expression of the positions after reset is then immediate.
\end{proof}

We prove Proposition~\ref{prop:region2}.

\begin{proof}[Proof of Proposition~\ref{prop:region2}]
By Lemma~\ref{lem:expr-reset}, the fixed point \[
\Bigl(\frac {v_1} {v_1+v_2+v_3}, \frac {v_1+v_2} {v_1+v_2+v_3} \Bigr) = \Bigl(\frac {r_1} {r_1+r_2+1}, \frac {r_1+r_2} {r_1+r_2+1} \Bigr)
\]
satisfies the condition~\eqref{fourth} of the definition of the reset function with strict inequalities. 
By assumption, we have $r_1 + 1 \leq r_2$. This implies $4r_1 \leq r_2^2$, which means that the discriminant of the characteristic polynomial of the matrix
$\begin{psmallmatrix}
    0 & -r_1 \\ 1 & -r_2
\end{psmallmatrix}$ 
is non-negative, and actually even positive because $r_1 < 1$. This matrix has thus two distinct real eigenvalues. One of these eigenvalues is smaller than $1$ in absolute value because the determinant of the matrix is equal to $r_1 < 1$. This implies the following: there is a segment containing the fixed point in its relative interior that is mapped to itself by the reset function and whose points form initial positions with convergence to the fixed point.
\end{proof}

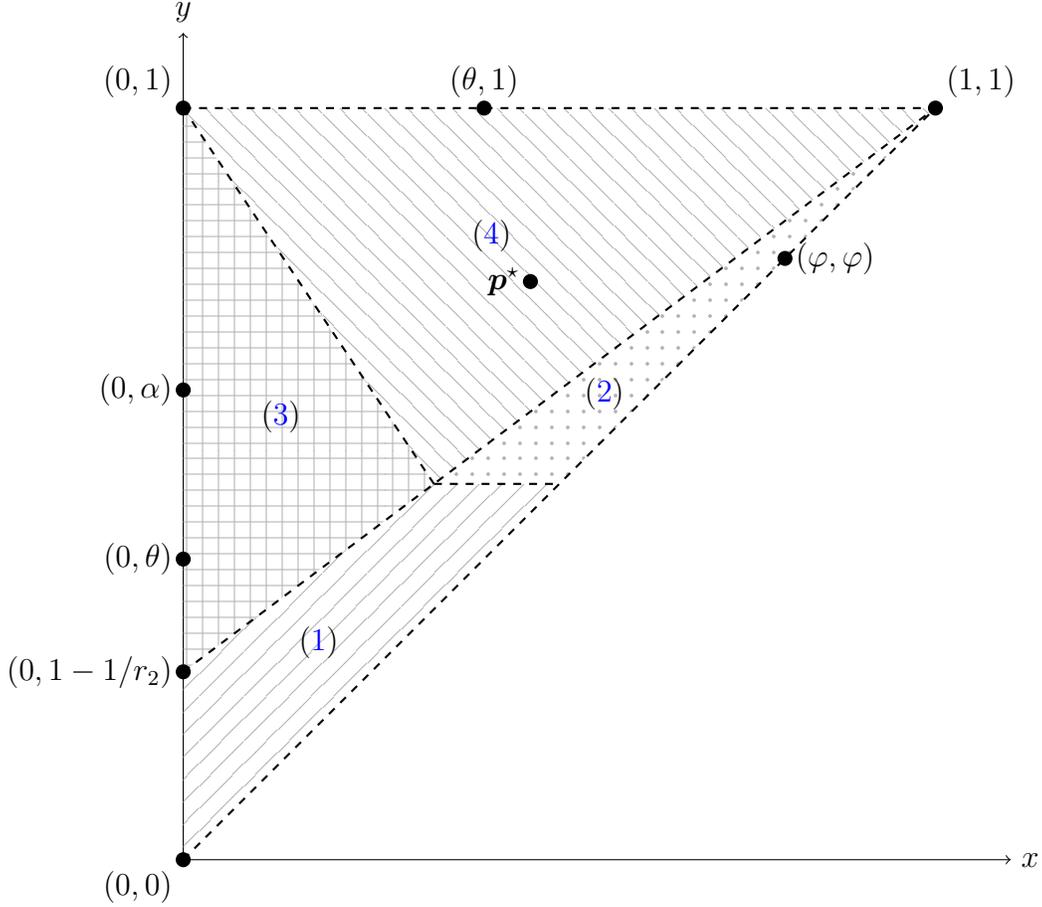
\begin{figure}
\begin{tikzpicture}[scale = 10]
  \draw[->] (0,0) -- (1.1,0) node[right] {$x$};
  \draw[->] (0,0) -- (0,1.1) node[above] {$y$};

  \fill[pattern={Lines[angle=45, distance=6pt]}, pattern color=black!30]
  (0,0) -- (0,1/4) -- (1/3,1/2) -- (1/2,1/2)-- cycle;

    \fill[pattern={Lines[angle=0, distance=6pt]}, pattern color=black!30]
  (0,1/4) -- (0,1) -- (1/3,1/2) -- cycle;
      \fill[pattern={Lines[angle=90, distance=6pt]}, pattern color=black!30]
  (0,1/4) -- (0,1) -- (1/3,1/2) -- cycle;

        \fill[pattern={Dots[radius=0.7pt, distance=6pt]}, pattern color=black!30]
  (1/3,1/2) -- (1,1) -- (1/2,1/2) -- cycle;
  
          \fill[pattern={Lines[angle=-45, distance=6pt]}, pattern color=black!30]
  (0,1) -- (1,1) -- (1/3,1/2) -- cycle;

  \draw[thick, dashed] (0,0) -- (1,1);
  
  \draw[thick, dashed] (1/3,1/2) -- (1/2,1/2);

   \draw[thick, dashed] (0,1) -- (1/3,1/2);

  \draw[thick, dashed] (0,1/4) -- (1,1);

  \draw[thick, dashed] (0,1) -- (1,1);

\node at (0.18,0.29) {$\eqref{first}$};

\node at (0.56,0.62) {$\eqref{second}$};

 \node at (0.13,0.59) {$\eqref{third}$};

\node at (0.41,0.83) {$\eqref{fourth}$};

  \fill (0,2/5) circle (0.01) node[left] {$(0,\theta)$};
    \fill (2/5,1) circle (0.01) node[above] {$(\theta,1)$};
      \fill (6/13,10/13) circle (0.01) node[left] {$\boldsymbol{p}^\star$};
  \fill (0,5/8) circle (0.01) node[left] {$(0,\alpha)$};
  \fill (4/5,4/5) circle (0.01) node[right] {$(\varphi,\varphi)$};
  \fill (0,0) circle (0.01) node[below left] {$(0,0)$};
    \fill (0,1) circle (0.01) node[above left] {$(0,1)$};
       \fill (1,1) circle (0.01) node[above right] {$(1,1)$};
          \fill (0,1/4) circle (0.01) node[left] {$(0,1-1/r_2)$};

\end{tikzpicture}
    \caption{\label{fig:321} Illustration of the proof of Theorem~\ref{thm:region3}. The position set $\triangle$ is subdivided into four cells corresponding to the behaviors~\eqref{first},~\eqref{second}, ~\eqref{third}, and~\eqref{fourth} of the reset map $f$. The point $\boldsymbol{p}^\star$ is the unique fixed point.
    }
\end{figure}

We give now the proof of Theorem~\ref{thm:region3}. Some elements of the proof are illustrated in Figure~\ref{fig:321}.

\begin{proof}[Proof of Theorem~\ref{thm:region3}]
We first establish that after finitely many iterations, we always reach a point of the form $(0,y)$, except when we start from the fixed point. To do so we proceed by a case distinction according to the four cases of the explicit expression of the reset function given by Lemma~\ref{lem:expr-reset}.

If some iteration reaches the case~\eqref{first}, then the next two iterations are of the form $(1,1)$, $(0,1)$. If some iteration reaches the case~\eqref{second}, then the next two iterations are of the form $(x,1)$, $(0,x)$. If some iteration reaches the case~\eqref{third}, then the next three iterations are of the form $(x,x)$, $(z,1)$, $(0,z)$. (For this latter case, we use $r_2 \geq 1$.)

We are left with the case~\eqref{fourth}. When $r_2^2-4r_1 \neq 0$, elementary computations show that the eigenvalues of the matrix $\begin{psmallmatrix} 0 & -r_1 \\ 1 & -r_2\end{psmallmatrix}$
have a modulus larger than $1$: 
\begin{itemize}
    \item in the case where $r_2^2-4r_1$ is negative, this is because the eigenvalues form a conjugate pair and have a product equal to $r_1 >1$; 
    \item  in the case where $r_2^2 -4r_1$ is positive, this results from the explicit computation of the eigenvalues that are then both real numbers (we have then $r_2 > 4$ and with $r_1 > r_2$, this leads to $r_2 - \sqrt{r_2^2-4r_1} > 2$).
\end{itemize}
 When $r_2^2-4r_1 = 0$, there is a unique (real) eigenvalue of norm larger than $1$ because its square is equal to $r_1 >1$.
All in all, this implies the following, assuming we start from a point distinct from the fixed point: if the iterates were always covered by case~\eqref{fourth}, then their norm would become arbitrarily large, contradicting the fact that the iterates remain in $\triangle$. Hence, the reset function cannot always be defined by case~\eqref{fourth}, and we come back to the previous cases.

\smallskip

We have therefore checked that for all possible initial positions, either we are on the fixed point, which corresponds to behavior~\ref{fixed}, or we reach a point of the form $(0,y)$ after finitely many resets. Now, we deal with three possible situations, depending on the relative position of $y$ with respect to $\theta = \frac{r_1(r_2-1)}{r_1r_2-1}$. We again rely on the explicit expression of the reset function given by Lemma~\ref{lem:expr-reset}.

Assume first that $y = \theta$. It is then straightforward to check that we are on a $3$-cycle:
\[
f(0,\theta) = (\theta,1) \, , \quad f(\theta, 1) = (\varphi,\varphi) \, , \quad f(\varphi,\varphi) = (0,\theta) \, .
\]
This corresponds to behavior~\ref{cycle-other}. 

Assume now $y < \theta$. We prove that after finitely many iterations, we reach a point of the form $(0,z)$ with $z \leq 1 - 1/r_2$. Suppose $(0,y)$ is not of this form. Since $r_2 < r_1$, we have $(1 - y)r_2 < \min(1, (1 - y)r_1)$. This places us in case~\eqref{third} of the reset function, so $f(0, y) = \bigl((1 - y)r_2, (1 - y)r_2\bigr)$. Since $y < \theta$, we are then in case~\eqref{second} and $f^2(0, y) = \bigl((1 - (1 - y)r_2)r_1, 1\bigr)$, and then in case~\eqref{fourth} and $f^3(0, y) = \bigl(0, (1 - (1 - y)r_2)r_1\bigr) = \bigl(0, \theta + (y - \theta)r_1r_2\bigr)$. Repeating the argument, we have $f^{3n}(0,y) = \bigl(0, \theta + (y-\theta)(r_1r_2)^n\bigr)$ as long as the second component is larger than $1 - 1/r_2$. Since $(r_1r_2)^n$ goes to $+\infty$ when $n\to +\infty$, there must exist an $n_0$ with $f^{3n_0}(0,y) = (0,z)$ of the desired form.
 We have then $f(0,z) = (1,1)$, again from the expression of the reset function. In other words, starting from $(0,y)$ with $y < \theta$, we reach behavior~\ref{cycle} after finitely many resets.

Assume finally that $y > \theta$. Restricted to the segment $[(0,\theta),(0,1)]$, the function $f$ is injective (given by case~\eqref{third}) and has its image located in the segment $[(0,0),(1,1)]$. Moreover, the image of $(0,1)$ is $(0,0)$, and the image of $(0,\theta)$ has its second coordinate larger than $1-1/r_1$. This implies that there is a unique point $(0,\alpha)$ in the segment $[(0,\theta),(0,1)]$ mapped to $(1-1/r_1, 1-1/r_1)$. If $y \geq \alpha$, then directly from the expression of the reset function (case~\eqref{first}), we get $f^2(0,y)=(1,1)$. So, we are left with the case where $y$ belongs to the interval $(\theta,\alpha)$. We have $f(0, y) = \bigl((1 - y)r_2, (1 - y)r_2\bigr)$, which places us in case~\eqref{second} ($y < \alpha$) and $f^2(0, y) = \bigl((1 - (1 - y)r_2)r_1, 1\bigr)$. We are then in case~\eqref{fourth} and $f^3(0, y) = \bigl(0, (1 - (1 - y)r_2)r_1\bigr) = \bigl(0, \theta + (y - \theta)r_1r_2\bigr)$. Repeating the argument, we have $f^{3n}(0,y) = (0, \theta + (y-\theta)(r_1r_2)^n)$, we have $f^{3n}(0,y) = (0, \theta + (y-\theta)(r_1r_2)^n)$ as long as the second component belongs to the interval $(\theta,\alpha)$. Since $(r_1r_2)^n$ goes to $+\infty$ when $n\to +\infty$, there must exist an $n_0$ with $f^{3n_0}(0,y) = (0,z)$ and $z \geq \alpha$. We have then $f(0,z) = (1,1)$, again from the expression of the reset function. In other words, starting from $(0,y)$ with $y > \theta$, we reach behavior~\ref{cycle} after finitely many resets.
\end{proof}

\section{Complementary discussion}\label{sec:compl}

\subsection{Case when \texorpdfstring{$v_3 < v_1 < v_2$}{v3 < v1 < v2}} In this case, even though we are still in Region~3, the behavior seems to be much more complicated. We still have initial positions for which the asymptotic behavior of Theorem~\ref{thm:region3} occurs. In particular, as already established by Bartholdi, Bunimovich, and Eisenstein, the following two asymptotic behaviors can be observed for some initial positions:
    \begin{enumerate}[label=\textup{(\roman*)}]
    \item\label{fixed-bis} it remains on the fixed point.
    \item\label{cycle-bis} it reaches after finitely many resets the standard $3$-cycle $(0,1,1)$, $(0,0,1)$, $(0,0,0)$.
    \end{enumerate}

They also established that behavior~\ref{cycle-bis} is the only stable cycle. Yet, experiments show that there are many initial positions for which the system reaches cycles, and even cycles of very large size, e.g., there are initial positions for which the cycle is of length $63{,}667$ when the speeds are given by $v_1 = 1.2$, $v_2 = 3$, $v_3 = 1$. This shows that not only ``Region~$k$ display the most interesting and complex dynamics.''

We emphasize that there is however some stability here, which was not discussed in Bartholdi, Bunimovich, and Eisenstein's paper. Indeed, there is a subset $\Sigma \subseteq \triangle$, which can occupy an arbitrarily large portion of $\triangle$ and in which the system stays when the initial position is taken in it. Moreover, since it is an open set for the induced topology on $\triangle$, small perturbations do not make the system leave $\Sigma$. This set $\Sigma$ is defined as follows.
Let
\[
\text{A} \coloneqq \bigl(0,r_1(1-\theta)\bigr) \, , \quad
\text{B} \coloneqq \bigl(r_1(1-\theta),r_1(1-\theta)\bigr)  \, , \quad
\text{C} \coloneqq \big(\theta,\theta\big)  \, , \quad
\text{D} \coloneqq (\theta / r_1,\theta)  \, ,
\]
\[
\text{E} \coloneqq \bigl(r_1(1-\theta),1\bigr)  \, , \quad
\text{F} \coloneqq \bigl(r_1(1-\theta),1\bigr)  \, , \quad
\text{G} \coloneqq (0,\theta)  \, , \quad
\text{H} \coloneqq (0,1-1/r_2)  \, ,
\]
and set
\[
\Sigma \coloneqq \Bigl(\conv\{\text{A},\text{B},\text{C},\text{D},\text{H}\} \cup \conv\{\text{E},\text{F},\text{G},\text{H},\text{D}\} \Bigr) \setminus \Bigl( [A,B] \cup [C,D] \cup [D,E] \cup [F,G] \Bigr) \, .
\]
From the expression of the reset function given in Section~\ref{subsec:three-workers}, it is then straightforward to check that $f(\Sigma) = \Sigma$. See Figure~\ref{fig:312} for an illustration.

\begin{figure}
\begin{tikzpicture}[scale = 10]
  \draw[->] (0,0) -- (1.1,0) node[right] {$x$};
  \draw[->] (0,0) -- (0,1.1) node[above] {$y$};

  \fill[pattern={Lines[angle=45, distance=6pt]}, pattern color=black!30]
  (0,0) -- (0,1/4) -- (1/4,1/4) -- cycle;

    \fill[pattern={Dots[radius=0.7pt, distance=6pt]}, pattern color=black!30]
  (0,1/4) -- (1/4,1/4) -- (1,1) -- (0,1/2) -- cycle;
  
        \fill[pattern={Lines[angle=-45, distance=6pt]}, pattern color=black!30]
  (1,1) -- (0,1/2) -- (0,1) -- cycle;

  \draw[thick, dashed] (0,0) -- (1,1);
  \draw[thick, dashed] (0,1/4) -- (1/4,1/4);

  \draw[thick, dashed] (0,1/2) -- (1,1);

  \draw[thick, dashed] (0,1) -- (1,1);

  \draw[thick] (0,4/15) -- (4/15,4/15);

  \draw[thick] (3/5,4/5) -- (4/5,4/5);

  \draw[thick] (3/5,4/5) -- (4/5,1);

  \draw[thick] (0,4/5) -- (4/15,1);

  \node at (0.09,0.17) {$\eqref{first}$};

  \node at (0.29,0.48) {$\eqref{second}$};

  \node at (0.36,0.84) {$\eqref{fourth}$};

  \fill (0,4/15) circle (0.01) node[left] {A};
  \fill (4/15,4/15) circle (0.01) node[below right] {B};
  \fill (4/5,1) circle (0.01) node[below right] {E};
  \fill (0,4/5) circle (0.01) node[below right] {G};
  \fill (4/15,1) circle (0.01) node[below right] {F};
  \fill (4/5,4/5) circle (0.01) node[below right] {C};
  \fill (3/5,4/5) circle (0.01) node[below right] {D};
  \fill (0,1/2) circle (0.01) node[below right] {H};
\end{tikzpicture}
    \caption{\label{fig:312} Illustration of the case~$v_3 < v_1 < v_2$. The position set $\triangle$ is subdivided into three cells corresponding to the behaviors~\eqref{first},~\eqref{second}, and~\eqref{fourth} of the reset map $f$ (in this case, behavior~\eqref{third} cannot occur). We have $f(A) = f(B) = E$, $f(C) = f(D) = F$, $f(E) =  G$, $f(F) = A$, $f(G) \in [A,B]$, and $f(H) \in (E,F)$.
    }
\end{figure}
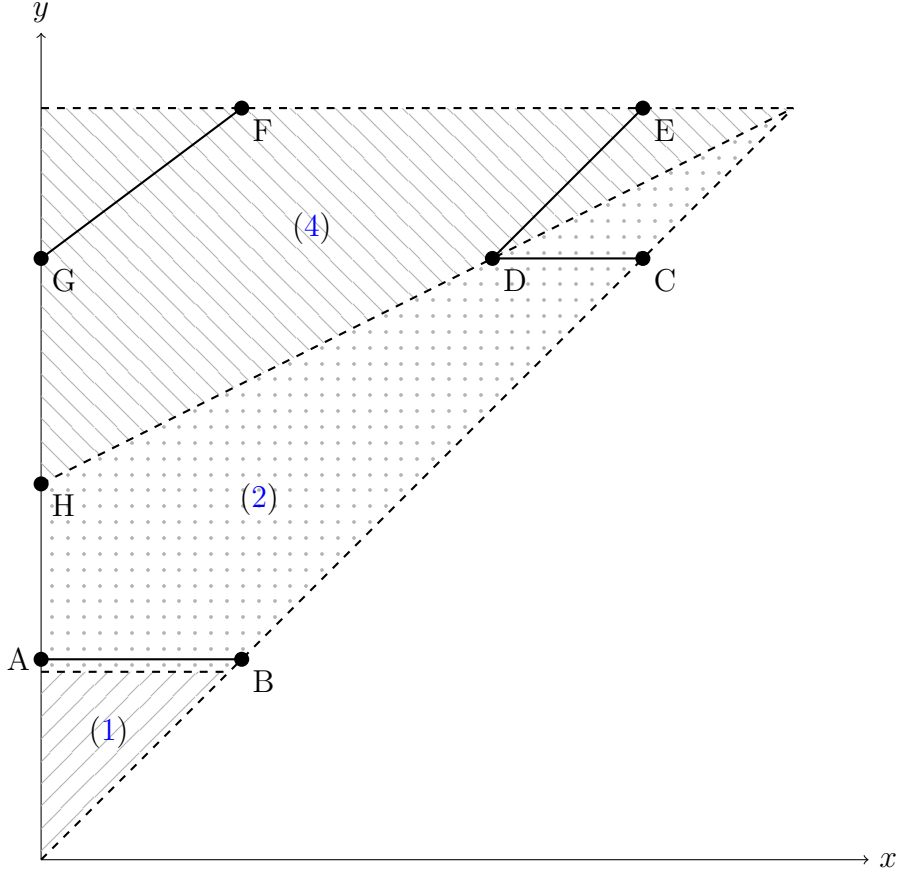

\subsection{Region~1}\label{subsec:region1}

Bartholdi, Bunimovich, and Eisenstein~\cite{Bartholdi1999} claim that the system always converges to the fixed point when the speeds lie in the ``Region 1,'' which is the set of speeds satisfying simultaneously $v_1 < v_3$ and $v_2 <  v_1 + v_3$. This is also confirmed by our own experiments. However, even though this is proved in the paper by Bartholdi and Eisenstein~\cite{Bartholdi1996APL} when $v_1 < v_2 < v_3$, they do not provide any proof for the other values in this region. This deserves future work.



\subsection{Four workers and beyond}

Even though we are able to describe quite precisely the asymptotic behavior of bucket brigades with three workers in many situations, we are far from providing a complete answer. Describing the asymptotic behavior of bucket brigades for four workers and more (still assuming constant velocities and instantaneous walk-balks) seem to be a challenging open problem.

\bibliographystyle{plain}
\bibliography{bucket-brigades}

\end{document}